\documentclass[11pt]{article}
\usepackage{amsfonts}
\usepackage{mathrsfs}
\usepackage{bbm}
\usepackage{amsfonts}
\usepackage{amssymb,amsmath,amsthm,graphicx}
\usepackage{float}
\usepackage[colorlinks, linkcolor=red]{hyperref}
\usepackage{color, xcolor}
\usepackage{cite}
\usepackage[paper=a4paper, left=1.6cm, right=1.6cm, top=1.8cm, bottom=1.6cm, headheight=5.5pt, footskip=0.8cm, footnotesep=0.8cm, centering, includefoot]{geometry}
\hypersetup{urlcolor=red, citecolor=red}

\DeclareMathOperator{\divv}{div}

\allowdisplaybreaks

\begin{document}
\title{Entropy-bounded solutions to the Cauchy problem of compressible planar non-resistive magnetohydrodynamics equations with far field vacuum
\thanks{
This research was partially supported by National Natural Science Foundation of China (Nos. 11971009, 11901288, 11901474, 11971307, 12071359), Scientific Research Foundation of Jilin Provincial Education Department (No. JJKH20210873KJ), Postdoctoral Science Foundation of
China (No. 2021M691219), and Exceptional Young Talents Project of Chongqing Talent (No. cstc2021ycjh-bgzxm0153).}
}
\author{Jinkai Li$\,^{\rm 1}\,$,\ Mingjie Li$\,^{\rm 2}\,$,\ Yang Liu$\,^{\rm 3, 4}\,$,\ Xin Zhong$\,^{\rm 5}\,${\thanks{Corresponding author. E-mail address: jklimath@m.scnu.edu.cn (J. Li), lmjmath@163.com (M. Li), liuyang0405@ccsfu.edu.cn (Y. Liu), xzhong1014@amss.ac.cn (X. Zhong).}}
\date{}\\
\footnotesize $^{\rm 1}\,$
School of Mathematical Sciences, South China Normal University, Guangzhou 510631, P. R. China\\
\footnotesize $^{\rm 2}\,$
College of Science, Minzu University of China, Beijing 100081, P. R. China\\
\footnotesize $^{\rm 3}\,$
School of Mathematics, Jilin University, Changchun 130012, P. R. China\\
\footnotesize $^{\rm 4}\,$ College of Mathematics, Changchun Normal
University, Changchun 130032, P. R. China\\
\footnotesize $^{\rm 5}\,$ School of Mathematics and Statistics, Southwest University, Chongqing 400715, P. R. China} \maketitle
\newtheorem{theorem}{Theorem}[section]
\newtheorem{definition}{Definition}[section]
\newtheorem{lemma}{Lemma}[section]
\newtheorem{proposition}{Proposition}[section]
\newtheorem{corollary}{Corollary}[section]
\newtheorem{remark}{Remark}[section]
\renewcommand{\theequation}{\thesection.\arabic{equation}}
\catcode`@=11 \@addtoreset{equation}{section} \catcode`@=12
\maketitle{}

\begin{abstract}
We investigate the Cauchy problem to the compressible planar non-resistive magnetohydrodynamic equations with zero heat conduction. The global
existence of strong solutions to such a model has been established by
Li and Li (J. Differential Equations 316: 136--157, 2022). However, to our best knowledge, so far there is no result on the behavior of the entropy near the vacuum region for this model. The main novelty of this paper is to give a positive response to this problem. More precisely, by a series of {\it a priori} estimates, especially the singular type estimates, we show that the boundedness of the entropy can be propagated up to any finite time provided that the initial vacuum presents only at far fields with sufficiently slow decay of the initial density.
%Moreover, we also obtain the $L^2$ regularities of the velocity and temperature.
%Our result generalizes Li and Xin's work (Adv. Math. 361: 106923, 2020) on the compressible Navier-Stokes system to the planar compressible non-resistive MHD flow although the interaction between the magnetohydrodynamics effects and hydrodynamic is complex.
\end{abstract}



%%%%%%%%%%%%%%%%%%%%%%%%%%%%%%%%%%%%%%%%%%%%%%%%%%%%%%%%%%%%%%%%%%%%%%%%%%%%%%%%%%%%%%%%%%%%%%%%%%

\section{Introduction}
\subsection{The planar compressible magnetohydrodynamic equations}
The time evolution of electrically conductive fluids in the presence of magnetic field is described
by the magnetohydrodynamic (MHD) equations. It is widely applied in astrophysics,
thermonuclear reactions, and industry, among many others. Assuming that the fluids are compressible, viscous, and heat conducting,
the governing equations can be read as the following system in Eulerian coordinates
\begin{align}\label{a1}
\begin{cases}
\rho_t+\divv(\rho u)=0,\\
(\rho u)_t+\divv(\rho u\otimes u)+\nabla P=(\nabla\times B)\times B+\divv\Psi(u),\\[3pt]
B_t-\nabla\times(u\times B)=-\nabla\times(\nu\nabla\times B), \quad \divv B=0,\\[3pt]
\Big(\mathcal E+\frac{|B|^2}{2}\Big)_t+\divv(u(\mathcal E+P))
=\divv\left((u\times B)\times B+\nu B\times(\nabla\times B)
+u\Psi(u)+\kappa\nabla\theta\right),
\end{cases}
\end{align}
where $\rho$, $u$, $P$, $B$, and $\theta$ denote the density,
velocity, pressure, magnetic field, and temperature, respectively.
\begin{align*}
\Psi(u)=\mu(\nabla u+\nabla^tu)+\lambda'\divv u\rm \mathbb{I}_3
\end{align*}
with $\mathbb{I}_3$ being the $3\times 3$ identity matrix, and $\nabla^tu$ being the transpose of the matrix $\nabla u$. $\mathcal E$ is the energy given by $\mathcal E:=\rho\big(e+\frac{|u|^2}{2}\big)$ with $e$
being the internal
energy.
%$\frac{\rho|u|^2}{2}$ is the kinetic energy, and $|B|^2/(8\pi)$ is the magnetic energy.
The viscosity coefficients $\mu$ and $\lambda'$ satisfy the physical restrictions
$$\mu>0,\ 2\mu+3\lambda'\ge 0.$$
$\nu\ge 0$ is the magnetic diffusion coefficient and $\kappa\ge 0$ is
 the heat conductivity. In this paper, we consider constitutive equations
 \begin{align}
 P=R\rho\theta, \quad e=c_v\theta,
 \end{align}
for two positive constants $R$ and $c_v$. Then, by the Gibbs equation $\theta Ds=De+PD(\frac{1}{\rho})$,
where $s$ is the specific entropy, one has the following relationship between $P$ and $s$:
\begin{align*}
P=Ae^{\frac{s}{c_v}}\rho^\gamma,\ A>0,\ \gamma=\frac{R}{c_v}+1>1.
\end{align*}

There is a considerable body of literature on compressible heat conducting MHD equations \eqref{a1} due to its physical importance, complexity, rich phenomena, and mathematical challenges.
See, for example, global (variational) weak solutions \cite{DF06,LG2014,HW2008}, global strong solutions with far field vacuum \cite{LZ20,HJP22,LZ22,LZ223}, asymptotic limits of solutions \cite{JJL12,JJLX14,JLL13}, and so on. We refer the interested reader to the
monograph \cite{LQ12}, which provides a detailed derivation of \eqref{a1} from the general constitutive laws, together with an extensive review of the mathematical theory and applications of this particular model.

In this paper, we consider a three-dimensional MHD flow \eqref{a1} with spatial variables $\mathbf{x}=(x, x_2, x_3)$ which is moving in the $x$ direction and uniform in the transverse direction $(x_2, x_3)$:
\begin{align}\label{a4}
\begin{cases}
\tilde{\rho}=\rho(x, t), \quad \tilde{\theta}=\theta(x, t), \\
u=(\tilde{u}, \tilde{w})(x, t), \quad \tilde{w}=(u_2, u_3),\\
B=(b_1, \tilde{b})(x, t), \quad \tilde{b}=(b_2, b_3),
\end{cases}
\end{align}
where $\tilde{u}$ and $b_1$ are the longitudinal velocity and longitudinal magnetic field, respectively, and $\tilde{w}$ and $\tilde{b}$ are the transverse velocity and transverse magnetic field, respectively. With this special structure \eqref{a4}, \eqref{a1} reduces to the following planar compressible magnetohydrodynamic flows with the constant
longitudinal magnetic field $b_1=1$ (without loss of generality) and $\lambda=\lambda'+2\mu>0$:
\begin{align}\label{a3}
\begin{cases}
\tilde{\rho}_t+(\tilde{\rho}\tilde{u})_x=0,\\[3pt]
(\tilde{\rho}\tilde{u})_t+(\tilde{\rho}\tilde{u}^2+\tilde{P})_x
=(\lambda\tilde{u}_x)_x-\tilde{b}\cdot \tilde{b}_x,\\[3pt]
(\tilde{\rho}\tilde{w})_t+(\tilde{\rho}\tilde{u}\tilde{w})_x-\tilde{b}_x
=\mu\tilde{w}_{xx},\\[3pt]
\tilde{b}_t+(\tilde{u}\tilde{b})_x-\tilde{w}_x=(\nu\tilde{b}_x)_x,\\[3pt]
\tilde{\mathcal E}_t+(\tilde{u}(\tilde{\mathcal E}+\tilde{P}+\frac12|\tilde{b}|^2)-\tilde{w}\cdot\tilde{b})_x
=(\lambda\tilde{u}\tilde{u}_x+\mu\tilde{w}\cdot\tilde{w}_x
+\nu\tilde{b}\cdot\tilde{b}_x+\kappa\tilde{\theta}_x)_x,
\end{cases}
\end{align}
where
\begin{align*}
\tilde{\mathcal E}=\tilde{\rho}\Big(\tilde{e}+\frac12(\tilde{u}^2+|\tilde{w}|^2)\Big)
+\frac12|\tilde{b}|^2,\ \tilde{P}=R\tilde{\rho}\tilde{\theta}=(\gamma-1)\tilde{\rho}\tilde{e}.
\end{align*}
In particular, when there are no magnetic diffusion and heat conductivity
(i.e., $\nu=\kappa=0$), \eqref{a3} becomes
\begin{align}\label{a2}
\begin{cases}
\tilde{\rho}_t+(\tilde{\rho}\tilde{u})_x=0,\\[3pt]
(\tilde{\rho}\tilde{u})_t+(\tilde{\rho}\tilde{u}^2+\tilde{P})_x
=(\lambda\tilde{u}_x)_x-\tilde{b}\cdot \tilde{b}_x,\\[3pt]
(\tilde{\rho}\tilde{w})_t+(\tilde{\rho}\tilde{u}\tilde{w})_x-\tilde{b}_x=\mu\tilde{w}_{xx},\\[3pt]
\tilde{b}_t+(\tilde{u}\tilde{b})_x-\tilde{w}_x=0,\\[3pt]
\tilde{P}_t+\tilde{u}\tilde{P}+\gamma\tilde{P}\tilde{u}_x
=(\gamma-1)\big(\lambda\tilde{u}_x^2
+\mu|\tilde{w}_x|^2\big).
\end{cases}
\end{align}

There is huge literature on the studies of \eqref{a3} with temperature-dependent heat-conductivity (i.e., $\kappa=\bar\kappa\tilde{\theta}^\beta$ with constants $\bar\kappa>0$ and $\beta\geq0$). In 2002, Chen and Wang \cite{CW02} established a global classical solution to a free boundary problem with large initial data for $\beta>2$. Wang \cite{W03} considered an initial-boundary value problem in a bounded spatial domain $(0,1)$ and proved global solutions with initial data in $H^1$. For $\beta>0$, Huang-Shi-Sun \cite{HSS19} showed global strong solutions to the initial-boundary-value problem in $(0,1)$ with large initial data, and they \cite{HSS21} also proved that such a solution is nonlinearly exponentially stable as time tends to infinity.
Meanwhile, Fan-Jiang-Nakamura \cite{FJN07} obtained the global existence of weak solution to the initial-boundary-value problem with large initial data when $\kappa(\tilde{\theta})\sim 1+\tilde{\theta}^q$ with $q\geq1$. This result was later extended by Fan-Huang-Li \cite{FHL17} to $q>0$. Recently, Li and Shang \cite{LS19} derived the global non-vacuum smooth solutions with large initial data when the viscosity, magnetic diffusion, and heat conductivity coefficients depend on the specific volume and temperature.
Li \cite{L18} proved global existence and uniqueness of strong solutions to \eqref{a3} with $\nu=0$ provided that the initial density is bounded below away from vacuum and the heat conductivity coefficient $\kappa(\tilde{\theta})\sim 1+\tilde{\theta}^\alpha$ for some $0<\alpha<\infty$. It should be pointed that the results mentioned above depend heavily on bounded domains and temperature-dependent heat-conductivity.
Recently, L{\"u}-Shi-Xiong \cite{LSX21} established global non-vacuum strong solutions to \eqref{a3} with constant viscosity and heat conductivity in unbounded domains.

Due to the lack of the expression of the entropy in the vacuum region and the
high singularity and degeneracy of the entropy equation close to the vacuum region,
in spite of its importance, the mathematical analysis of the entropy for the viscous
compressible fluids in the presence of vacuum was a challenge to study its dynamics.
As stated above, the huge study of planar non-resistive magnetohydrodynamics equations are mainly focused on the well-posedness. However, there is no result on the behavior of the entropy near the vacuum region for this model.
As our fist step in our series results, we begin with a simple case \eqref{a2}.
Very recently, the global well-posedness of strong solutions of \eqref{a2} with large initial data and vacuum in $\mathbb{R}$ was obtained in our previous work \cite{LL22}. In this paper, we continue our study on the uniform boundedness of the
entropy for the model \eqref{a2} in the presence of vacuum.

\subsection{Reformulation in Lagrangian coordinates and main result}
Let $y$ be the Lagrangian coordinate and define the coordinate transform between $y$ and
the Euler coordinate $x$ as
\begin{align*}
x=\eta(y, t),
\end{align*}
where $\eta(y, t)$ is the flow map determined by $\tilde{u}$, that is
\begin{align}\label{a5}
\begin{cases}
\partial_t\eta(y, t)=\tilde{u}(\eta(y, t), t),\\
\eta(y, 0)=y.
\end{cases}
\end{align}
Denote
\begin{align}\label{a6}
\begin{cases}
\rho(y, t)=\tilde{\rho}(\eta(y, t), t), \quad
u(y, t)=\tilde{u}(\eta(y, t), t), \quad w(y, t)=\tilde{w}(\eta(y, t), t),\\
h(y, t)=\tilde{b}(\eta(y, t), t), \quad P(y, t)=\tilde{P}(\eta(y, t), t).
\end{cases}
\end{align}
Then it is not hard to check that
\begin{align*}
&(\tilde u_x, \tilde{w}_x, \tilde{b}_x, \tilde{P}_x)
=\left(\frac{u_y}{\eta_y}, \frac{w_y}{\eta_y}, \frac{h_y}{\eta_y},
\frac{P_y}{\eta_y}\right),\quad (\tilde{u}_{xx}, \tilde{w}_{xx})=\bigg(\frac{1}{\eta_y}\left(\frac{u_y}{\eta_y}\right)_y,
\frac{1}{\eta_y}\left(\frac{w_y}{\eta_y}\right)_y\bigg),\\[3pt]
&\tilde{\rho}_t+\tilde{u}\tilde{\rho}_x=\rho_t,
\, \tilde{u}_t+\tilde{u}\tilde{u}_x=u_t, \, \tilde{w}_t+\tilde{u}\tilde{w}_x=w_t, \,
\tilde{b}_t+\tilde{u}\tilde{b}_x=h_t, \, \tilde{P}_t+\tilde{u}\tilde P_x=P_t.
\end{align*}
Define a function $J=J(y, t)$ as
\begin{align*}
J(y, t)=\eta_y(y, t),
\end{align*}
then it follows from \eqref{a5} and \eqref{a6} that
\begin{align}\label{1.5}
J_t=u_y.
\end{align}
Thus, \eqref{a2} can be rewritten in the Lagrangian coordinate as
\begin{align}\label{1.6}
\begin{cases}
\rho_t+\frac{u_y}{J}\rho=0,\\[3pt]
\rho u_t-\frac{\lambda}{J}\Big(\frac{u_y}{J}\Big)_y+\frac{P_y}{J}+\frac{1}{ J}h\cdot h_y=0,\\[3pt]
\rho w_t-\frac{\mu}{J}\Big(\frac{w_y}{J}\Big)_y=\frac{1}{J}h_y,\\[3pt]
h_t+\frac{u_y}{J}h-\frac{w_y}{J}=0,\\[3pt]
P_t+\gamma\frac{u_y}{J}P=(\gamma-1)\Big(\lambda\left|\frac{u_y}{J}\right|^2+\mu\left|\frac{w_y}{J}\right|^2\Big).
\end{cases}
\end{align}
Due to \eqref{1.5} and $\eqref{1.6}_1$, it holds that
\begin{align*}
(J\rho)_t=J_t\rho+J\rho_t=u_y\rho-J\frac{u_y}{J}\rho=0,
\end{align*}
from which, by setting $\rho|_{t=0}=\rho_0$ and noticing that $J|_{t=0}=1$, we obtain that
\begin{align*}
J\rho=\rho_0.
\end{align*}
Therefore, one can rewrite \eqref{1.6} as
\begin{align}\label{1.7}
\begin{cases}
J_t=u_y,\\[3pt]
\rho_0u_t-\lambda\big(\frac{u_y}{J}\big)_y+P_y+h\cdot h_y=0,\\[3pt]
\rho_0w_t-\mu\big(\frac{w_y}{J}\big)_y=h_y,\\[3pt]
h_t+\frac{u_y}{J}h-\frac{w_y}{J}=0,\\[3pt]
P_t+\gamma\frac{u_y}{J}P
=(\gamma-1)\big(\lambda\left|\frac{u_y}{J}\right|^2
+\mu\left|\frac{w_y}{J}\right|^2\big).
\end{cases}
\end{align}

We will consider the Cauchy problem and, thus, complement system \eqref{1.7} with the initial condition
\begin{align}\label{1.8}
(J, \sqrt{\rho_0}u, \sqrt{\rho_0}w, h, P)|_{t=0}=(J_0,\sqrt{\rho_0}u_0, \sqrt{\rho_0}w_0, h_0, P_0),
\end{align}
where $J_0$ has uniform positive lower and upper bounds. It should be pointed out that,
by the definition of $J$, the initial $J_0$ should be identically one;  however, in order to extend the local solution to be a
global one, one may take some positive time $T_*$ as the initial time at which $J$ is
not necessary to be identically one. As a result, we have to deal with the local
well-posedness result with initial $J_0$ not being identically one.

The following conventions will be used throughout this paper. For $1\le q\le \infty$
and positive integer $m$, $L^q=L^q(\mathbb{R})$ and $W^{m, q}=W^{m, q}(\mathbb{R})$
denote the standard
Lebesgue and Sobolev spaces, respectively, and $H^m=W^{m, 2}$. For simplicity,
$L^q$ and $H^m$ denote also their $N$ product spaces $(L^q)^N$ and $(H^m)^N$, respectively.
$\|u\|_q$ is the $L^q$ norm of $u$, and $\|(f_1, f_2, \ldots, f_n)\|_X$ is the sum $\sum_{i=1}^N\|f_i\|_X$ or
the equivalent norm $(\sum_{i=1}^N\|f_i\|_X^2)^{\frac12}$. Moreover, we write
\begin{align*}
\int \cdot dy=\int_{\mathbb{R}}\cdot dy.
\end{align*}

Local and global strong solutions to the problem \eqref{1.7}--\eqref{1.8} are defined in the
following two definitions.
\begin{definition}
Given a positive time $T\in (0, \infty)$. A quintuple $(J, u, w, h, P)$ is called a strong solution
to the problem \eqref{1.7}--\eqref{1.8}, on $\mathbb{R}\times (0, T)$, if it has the properties
\begin{align*}
&\inf_{y\in \mathbb{R}, t\in (0, T)}J(y, t)>0, \quad P\ge 0 \ \text{\rm on}~\mathbb{R}\times(0, T),\\
&J-J_0\in C([0, T]; H^1), \quad \frac{J_y}{\sqrt{\rho_0}}\in L^\infty(0, T; L^2),
\quad J_t\in L^\infty(0, T; L^2),\\
&(\sqrt{\rho_0}u, \sqrt{\rho_0}w)\in C([0, T]; L^2), \quad (u_y, w_y)\in L^\infty(0, T; L^2), \\
&\left(\sqrt{\rho_0}u_t, \sqrt{\rho_0}w_t, \frac{u_{yy}}{\sqrt{\rho_0}}, \frac{w_{yy}}{\sqrt{\rho_0}}\right)
\in L^2(0, T; L^2),\quad h_t\in L^\infty(0, T; L^2),\\
&(P, h)\in C([0, T]; L^2), \quad \left(\frac{P_y}{\sqrt{\rho_0}}, \frac{h_y}{\sqrt{\rho_0}}\right)
\in L^\infty(0, T; L^2), \quad P_t\in L^4(0, T; L^2),
\end{align*}
satisfies equations \eqref{1.7}, a.e. in $\mathbb{R}\times (0, T)$, and fulfills the initial condition \eqref{1.8}.
\end{definition}

\begin{definition}
A quintuple $(J, u, w, h, P)$ is called a global strong solution to the problem \eqref{1.7}--\eqref{1.8},
if it is a strong solution to the same system on $\mathbb{R}\times(0, T)$, for any positive time $T\in (0, \infty)$.
\end{definition}

Due to \eqref{1.7}$_4$, we see that the equation satisfied by $|h|^2$ has the similar structure as \eqref{1.7}$_5$. Thus, we can obtain the following local existence of strong solutions to system \eqref{1.7} with initial condition \eqref{1.8}, which can be proved in the similar way as those in \cite{LX20,LX22}.
\begin{theorem}[Local well-posedness]\label{thm1}
Assume that $(\rho_0, u_0, w_0, h_0, P_0)$ satisfies
\begin{align*}
&\inf_{y\in (-r, r)}\rho_0(y)>0, \quad
\forall r\in (0, \infty), \quad \rho_0\le \bar{\rho} ~on~\mathbb{R},\label{h1}\tag{$\rm H1$}\\
&\left(\sqrt{\rho_0}u_0, \sqrt{\rho_0}w_0, u_0', w_0', h_0, P_0, \frac{P_0'}{\sqrt{\rho_0}},
\frac{h_0'}{\sqrt{\rho_0}}\right)\in L^2,\quad P_0\ge 0~on~\mathbb{R} \label{h2}\tag{$\rm H2$},\\
&\underline{J}\le J_0\le \bar{J}~on~\mathbb{R}, \quad \frac{J_0'}{\sqrt{\rho_0}}\in L^2,
\end{align*}
for positive constants $\bar{\rho}$, $\underline{J}$, and $\bar{J}$. Denoting $F_0:=\mu w_0'-h_0$,
 $G_0:=\lambda u_0'-P_0-\frac{|h_0|^2}{2}$, and $H_0=|h_0|^2$,
then the followings hold.
\begin{description}
  \item{\rm(i)} There is a positive time $T$ depending only on $\gamma$, $\lambda$, $\mu$, $\bar{\rho}$,
  $\|u_0'\|_2$, $\|w_0'\|_2$, $\|P_0\|_2$, $\|h_0\|_4$, $\|h_0\|_\infty$, and $\|P_0\|_\infty$,
  such that system \eqref{1.7}, subject to the initial condition \eqref{1.8}, has
   unique strong solution $(J, u, w, h, P)$, on $\mathbb{R}\times (0, T)$.
  \item{\rm(ii)} Assume in addition that
  \begin{align}\label{h3}\tag{$\rm H3$}
  |\rho_0'(y)|\le K_1\rho_0^\frac32, \quad \forall y\in \mathbb{R}, \quad  (\rho_0^{-\frac{\alpha}{2}}F_0,
  \rho_0^{-\frac{\alpha}{2}}G_0, \rho_0^{-\frac{\alpha}{2}}H_0)\in L^2,
  \quad (\rho_0^{-\frac{\alpha}{2}}P_0,
  \rho_0^{-\frac{\alpha}{2}}h_0)\in H^1,
  \end{align}
  for two positive constants $\alpha$ and $K_1$.
\end{description}

Then, $(J, u, w, h, P)$ has the additional regularities
\begin{align}
\begin{cases}
\left(\rho_0^{-\frac{\alpha}{2}}F, \rho_0^{-\frac{\alpha}{2}}G\right)\in L^\infty(0, T; L^2)\cap L^4(0, T;
L^\infty),\quad \left(\rho_0^{-\frac{\alpha+1}{2}}F, \rho_0^{-\frac{\alpha+1}{2}}G\right)\in L^2(0, T; L^2),\label{1.9}\\
~\rho_0^{-\frac{\alpha}{2}}H\in L^\infty(0, T; L^2)\cap L^2(0, T; L^2),
\quad \rho_0^{-\frac{\alpha}{2}}H^\frac32\in L^2(0, T; L^2), \quad\rho_0^{-\frac{\alpha}{2}}\sqrt{P}H\in L^2(0, T; L^2),
\end{cases}
\end{align}
where $H:=|h|^2$ is the transverse magnetic field, $G:=\lambda\frac{u_y}{J}-P-\frac{H}{2}$ is the effective viscous flux, $F:=\mu\frac{w_y}{J}+h$ is the transverse effective viscous flux,
and
\begin{align}\label{1.11}
\begin{cases}
u\in L^\infty(0, T; H^1), \quad \text{if}~u_0\in H^1~\text{and}~\alpha\ge 1,\\
w\in L^\infty(0, T; H^1), \quad \text{if}~w_0\in H^1~\text{and}~\alpha\ge 1,\\
\theta\in L^\infty(0, T; H^1), \quad \text{if}~\theta_0\in H^1,
\frac{J_0'}{\rho_0}\in L^2,~\text{and}~\alpha\ge 1,\\
s\in L^\infty(0, T; L^\infty),\quad \text{if}~s_0\in L^\infty~\text{and}~\alpha\ge \gamma,
\end{cases}
\end{align}
where $\theta:=\frac{P}{R\rho}$ and $s:=c_v\ln\big(\frac{P}{A\rho^\gamma}\big)$, respectively,
 are the corresponding temperature
and entropy, with $\rho:=\frac{\rho_0}{J}$, $\theta_0:=\frac{P_0}{R\rho_0}$,
and $s_0:=c_v\ln\big(\frac{P_0}{A\rho_0^\gamma}\big)$ being the density,
the initial temperature, and the initial entropy, respectively.
\end{theorem}

The main result of this paper reads as follows.
\begin{theorem}\label{thm2}
For the two positive constants $\alpha$ and $K_1$, let \eqref{h1}--\eqref{h3} be satisfied,  and further assume that
\begin{align}\label{h4}\tag{$\rm H4$}
\rho_0\in L^1, \quad P_0\in L^1.
\end{align}
  Then, \eqref{1.9} and \eqref{1.11} hold for any $T\in (0, \infty)$.

\end{theorem}

\begin{remark}
We remove the following key assumption of initial density used in \cite{LX20}
\begin{align}\label{s1}\tag{$*$}
\rho_0(y)\ge \frac{A_0}{(1+|y|)^2}, \quad \forall y\in\mathbb{R},
\end{align}
for some positive constant $A_0$, which is very important to control the term $|u(y, t)|$ by $\|\sqrt{\rho_0}u\|_2$.
In particular, Theorem \ref{thm2} gives a positive answer to the dynamics of the entropy for planar non-resistive magnetohydrodynamics equations
without heat conduction.
\end{remark}

\begin{remark}
For the the heat conductive
case \eqref{a3}, one may only need to deal with the the far field vacuum, as the heat conductivity will make the temperature strictly positive everywhere after the initial time,
which implies that the entropy becomes unbounded instantaneously if the interior
vacuum occurs initially. However, positive heat conductivity leads to both increase
and decrease of the entropy, and the nonlinearity term and coupling term of the model \eqref{a3} make equations  more complicated, and both
creates substantial difficulties in the analysis
compared with the case \eqref{a2}. We will solve this problem in the future work.
\end{remark}

\begin{remark}
The condition $ |\rho_0'(y)|\le K_1\rho_0^\frac32$ in \eqref{h3} is essentially slow decay assumption on $\rho_0$ at the far field.
In fact, for $\rho_0(y)=\frac{K_\rho}{\langle y\rangle^{\ell_\rho}}$ with $\langle y\rangle=(1+y^2)^\frac12$ and positive constants
$K_\rho$ and $\ell_\rho$, it holds that
\begin{align*}
|\rho_0'(y)|\le K_1\rho_0^\frac32\Leftrightarrow 0\le \ell_\rho\le 2.
\end{align*}
\end{remark}

\begin{remark}
Let $K_\rho$ and $1<\ell_\rho\le 2$ be positive constants. Choose
\begin{align*}
\rho_0(y)=\frac{K_\rho}{\langle y\rangle^{\ell_\rho}}, \quad J_0\equiv1, \quad u_0, w_0, b_0\in C_0^\infty,
\quad P_0=Ae^\frac{1}{c_v}\rho_0^\gamma.
\end{align*}
Then, one can verify easily that \eqref{h1}--\eqref{h4} hold. Therefore, the set of the
initial data that fulfills the conditions in the above theorems is not empty.
\end{remark}

We now sketch the main idea used in this paper.
In our case (i.e., $\kappa=\nu=0$), the entropy can only increase in time
and thus it is bounded from below trivially. Thus, the key issue in our proof is to derive the upper bound of the entropy. To this end, similar to \cite{LX20} for the existence of entropy bounded solutions to the compressible Navier-Stokes equations with zero heat conduction, we should establish the appropriate singular type estimates, up to any finite time, of the solutions to system \eqref{1.7}, subject to \eqref{1.8}. However, due to the absence of \eqref{s1}, some new ideas are needed.

There are two main stages for carrying out the desired singular type
estimates. In the first stage, noticing that the effective viscous flux
\begin{align*}
G:=\lambda\frac{u_y}{J}-P-\frac{|h|^2}{2}
\end{align*}
satisfies
\begin{align}\label{w1}
G_t-\frac{\lambda}{J}\Big(\frac{G_y}{\rho_0}\Big)_y
=-\gamma\frac{u_y}{J}G+\frac{2-\gamma}{2}\frac{u_yH}{J}-(\gamma-1)\mu\Big|\frac{w_y}{J}\Big|^2-\frac{h\cdot w_y}{J},
\end{align}
%and working on its $L^\infty(L^2)\cap L^2(H^1)$ type {\it a priori} estimate, one can obtain {\it a priori} $L^\infty(H^1)$
%estimate on $(J, u, w, h)$. However,
the integrals related to terms except $\gamma\frac{u_y}{J}G$ in \eqref{w1}
cannot be handled by the same idea as in \cite{LX20}.  To deal with other terms in \eqref{w1}, motivated by \cite{LL22}, we introduce a new vector field
$F$, called ``transverse effect viscous flux", as
 \begin{align*}
F:=\mu\frac{w_y}{J}+h.
\end{align*}
One can check that $F$ satisfies
\begin{align}\label{w2}
F_t-\frac{\mu}{J}\left(\frac{F_y}{\rho_0}\right)_y=-\frac{u_y}{J}F
+\frac{w_y}{J}.
\end{align}
Based on the above facts, in order to obtain the estimate of $\sup_{0\le \tau\le t}\|\rho_0^{-\frac{\alpha}{2}}G\|_2$ and $\int_0^t\int\frac{(G_y)^2}{\rho_0^{\alpha+1}}dyd\tau$, we make use of $\frac{\rho_0}{\rho_\delta^{\alpha+1}}JG\varphi_r^2$ ($\rho_\delta:=\rho_0+\delta$) instead of $\frac{JG\varphi_r^2}{\rho_0^\alpha}$ to test \eqref{w2} in order to overcome the defect of the assumption \eqref{s1} (see \eqref{2.62}).
%In fact, the key estimate to control the term of
% $\int_0^t\int_{r\le |y|\le 2r}\frac{|G||G_y|}{\rho_0}\varphi_r'dyd\tau$ as $r\rightarrow \infty$.
 In \cite{LX20}, by the assumption \eqref{s1}, one has that
\begin{align*}
\left|\int_0^t\int_{r\le |y|\le 2r}\frac{G^2}{\rho_0^{\alpha+1}}|\varphi_r'|^2dyd\tau\right|
\le M_0^2\int_0^t\int_{r\le |y|\le 2r}\frac{G^2}{\rho_0^\alpha}dyd\tau\rightarrow 0, \quad \text{as}~r\rightarrow \infty,
\end{align*}
where
\begin{align*}
\left|\varphi_r'(y)\right|\le M_0\sqrt{\rho_0(y)}, \quad \forall y\in\Bbb R, \quad M_0=\frac{4\|\eta'\|_\infty}{\sqrt{A_0}},
\end{align*}
for any $r\ge 1$ and $t\in [0, T]$.
However, if we neglect the assumption \eqref{s1}, it is difficult to control $\int_0^t\int_{r\le |y|\le 2r}\frac{G^2}{\rho_0^{\alpha+1}}|\varphi_r'|^2dyd\tau$ directly because $\rho_0$ may contain vacuum at the far field. In our case, we consider $\frac{C}{r^2}\int_0^t\int_{r\le |y|\le 2r}\frac{\rho_0}{\rho_\delta^{\alpha+1}}G^2dyd\tau$ instead of $\int_0^t\int_{r\le |y|\le 2r}\frac{G^2}{\rho_0^{\alpha+1}}|\varphi_r'|^2dyd\tau$ and we show that, for any $\delta>0$,
\begin{align*}
\frac{C}{r^2}\int_0^t\int_{r\le |y|\le 2r}\frac{\rho_0}{\rho_\delta^{\alpha+1}}G^2dyd\tau
\le \frac{C\delta^{-1}}{r}\int_0^t\int_{r\le |y|\le 2r}\frac{G^2}{\rho_\delta^\alpha}dyd\tau\rightarrow 0, \quad\text{as}~r\rightarrow\infty.
\end{align*}
Similarly, one can also obtain that $\int_0^t\int_{r\le |y|\le 2r}\frac{|F|^2}{\rho_\delta^\alpha}|\varphi_r'|^2dyd\tau\rightarrow 0$ as $r\rightarrow \infty$.
After that, by a direct calculation, one can get the $L^\infty(0,T; L^2)$ estimate of $(JF, JG, JH)$ through testing \eqref{2.17}, \eqref{2.30},
 \eqref{2.34} with $\frac{\rho_0}{\rho_\delta^{\alpha+1}}JF\varphi_r^2$,
$\frac{\rho_0JH}{\rho_\delta^{\alpha+1}}\varphi_r^2$, $\frac{\rho_0}{\rho_\delta^{\alpha+1}}JG\varphi_r^2$, respectively.
With this {\it a priori} estimate on $\rho_0^{-\frac{\alpha}{2}}F$ at hand, and combining the $L^2$ type estimate $\rho_0^{-\frac{\alpha}{2}}G$ with $L^4$ type estimate on $\rho_0^{-\frac{\alpha}{4}}h$, one gets the desired $L^4(0,T;L^\infty)$ type estimate on $\rho_0^{-\frac{\alpha}{2}}G$ and $\rho_0^{-\frac{\alpha}{2}}F$. Indeed, we carry out the $L^2$-type estimate on $\rho_0^{-\frac{\alpha}{2}}h$,
which are achieved through performing $L^4(0, T; L^\infty)$ type estimate on $(J, u, w, h, P)$ together with $L^2$-type estimates on $\rho_0^{-\frac{\alpha}{2}}F$.
  With these estimates in hand, we can show the desired upper bound and lower bound of the
entropy (see Lemma \ref{l28}).

The rest of this paper is arranged as follow. In Section \ref{sec2}, we derive
important singularly weighted \textit{a priori} estimates. Section \ref{sec3} is devoted to proving Theorem \ref{thm2}; however, since the solutions being established are Lipschitz-continuous, all results can be transformed accordingly in the Euler coordinates.

\section{\textit{A priori} estimates}\label{sec2}
This section is devoted to deriving a series of {\it a priori}
estimates, which are finite up to any finite time. Throughout this section, we always assume that $J_0\equiv 1$.

We start with the following a {\it priori} estimates of global strong solutions, which can be found in our previous work \cite{LL22}.
\begin{lemma}\label{l21}
Assume that \eqref{h1}--\eqref{h2}, and \eqref{h4} are satisfied, it holds that
\begin{align}
&\int\left(\frac{\rho_0u^2}{2}+\frac{\rho_0|w|^2}{2}+\frac{J|h|^2}{2}+\frac{JP}{\gamma-1}\right)dy=\mathcal E_0,
\quad \inf_{y\in\Bbb R}J(y, t)\ge e^{-\frac{2\sqrt{2}}{\lambda}\sqrt{\mathcal E_0}\|\rho_0\|_1}, \\
& \sup_{0\le t\le T}\Big\|\Big(J-1, F, G, H, \frac{J_y}{\sqrt{\rho_0}}, J_t, h_t, \frac{h_y}{\sqrt{\rho_0}},
 u_y, w_y, P, \frac{P_y}{\sqrt{\rho_0}}\Big)\Big\|_2 \nonumber\\
& \quad +
\int_0^T\Big\|\Big(\frac{F_y}{\sqrt{\rho_0}}, \frac{G_y}{\sqrt{\rho_0}}, \frac{u_{yy}}{\sqrt{\rho_0}},
\frac{w_{yy}}{\sqrt{\rho_0}}, \sqrt{\rho_0}u_t, \sqrt{\rho_0}w_t, H^\frac32, \sqrt{P}H\Big)\Big\|_2^2dt\nonumber\\
&\quad+\int_0^T\Big(\|P_t\|_2^4+\|(F, G)\|_\infty^4\Big)dt\le C,\label{t2.2}
\end{align}
where
\begin{align*}
\mathcal E_0:=\int\Big(\frac{\rho_0u_0^2}{2}+\frac{\rho_0|w_0|^2}{2}+\frac{|h_0|^2}{2}
+\frac{P_0}{\gamma-1}\Big)dy,
\end{align*}
and
\begin{align*}
F:=\mu\frac{w_y}{J}+h,\quad  H:=|h|^2, \quad G:=\lambda\frac{u_y}{J}-P-\frac{|h|^2}{2}
=\lambda\frac{u_y}{J}-P-\frac{H}{2}.
\end{align*}

\end{lemma}

The conservation of the momentum is stated in the following lemma.
\begin{lemma}\label{l23}
It holds that
\begin{align*}
\int\rho_0(y)u(y, t)dy=\int\rho_0u_0(y)dy,\quad \int\rho_0(y)w(y, t)dy=\int\rho_0(y)w_0(y)dy,\quad \forall t\in[0, T].
\end{align*}
\end{lemma}
\begin{proof}[Proof]
Multiplying $\eqref{1.7}_{2,3}$ by $\varphi_r$ and integrating the resultants over $\mathbb{R}$, one gets that
\begin{align*}
&\frac{d}{dt}\int \rho_0u\varphi_rdy=\int\Big(P+\frac{1}{2}|h|^2-\lambda\frac{u_y}{J}\Big)\varphi_r'dy,\\
&\frac{d}{dt}\int \rho_0w\varphi_rdy=\int\Big(-h-\mu\frac{w_y}{J}\Big)\varphi_r'dy,
\end{align*}
which imply that
\begin{align*}
&\int \rho_0u\varphi_rdy=\int\rho_0u_0\varphi_rdy
+\int_0^t\int\Big(P+\frac{|h|^2}{2}-\lambda\frac{u_y}{J}\Big)\varphi_r'dyd\tau,\\
&\int \rho_0w\varphi_rdy=\int\rho_0w_0\varphi_rdy
-\int_0^t\int\Big(h+\mu\frac{w_y}{J}\Big)\varphi_r'dyd\tau.
\end{align*}
We claim that
\begin{align}
&\lim_{r\rightarrow \infty}\int_0^t\int\Big(P+\frac{1}{2}|h|^2-\lambda\frac{u_y}{J}\Big)\varphi_r'dyd\tau=0,\label{2.14}\\
&\lim_{r\rightarrow \infty}\int_0^t\int\Big(h+\mu\frac{w_y}{J}\Big)\varphi_r'dyd\tau=0.\label{2.15}
\end{align}
In fact,
by Lemma \ref{l21},
$P\in L^1(\mathbb{R}\times (0, T))$, $(u_y, w_y)\in L^2(\mathbb{R}\times (0, T))$, and $h\in L^\infty(0, T; L^2(\mathbb{R}))$,
we derive from H\"older's inequality that
\begin{align*}
&\Big|\int_0^t\int\Big(P+\frac{1}{2}|h|^2-\lambda\frac{u_y}{J}\Big)\varphi_r'dyd\tau\Big|\nonumber\\
&\le \frac{C}{r}\int_0^t\int_{r\le|y|\le 2r}\Big|P+\frac{1}{2}|h|^2-\lambda\frac{u_y}{J}\Big|dyd\tau\nonumber\\
&\le \frac{C}{r}\int_0^t\int\Big(P+\frac{1}{2}|h|^2\Big)dyd\tau
+\frac{Cc_0^{-1}}{\sqrt{r}}\int_0^t\Big(\int u_y^2dy\Big)^\frac12d\tau\nonumber\\
&\le \frac{C}{r}\int_0^t\int\Big(P+\frac{1}{2}|h|^2\Big)dyd\tau
+\frac{Cc_0^{-1}}{\sqrt{r}}\Big(\int_0^t\int u_y^2dyd\tau\Big)^\frac12\rightarrow 0, \quad \text{as}~r\rightarrow\infty,
\end{align*}
and
\begin{align*}
&\Big|\int_0^t\int\Big(h+\mu\frac{w_y}{J}\Big)\varphi_r'dyd\tau\Big|\nonumber\\
&\le \frac{C}{r}\int_0^t\int_{r\le|y|\le 2r}\Big|h+\mu\frac{w_y}{J}\Big|dyd\tau\nonumber\\
&\le \frac{C}{\sqrt{r}}\int_0^t\Big(\int_{r\le|y|\le 2r}|h|^2dy\Big)^\frac12d\tau
+\frac{Cc_0^{-1}}{\sqrt{r}}\int_0^t\Big(\int_{r\le|y|\le 2r}|w_y|^2dy\Big)^\frac12d\tau\nonumber\\
&\le \frac{C}{\sqrt{r}}\int_0^t\Big(\int_{r\le|y|\le 2r}|h|^2dy\Big)^\frac12d\tau
+\frac{Cc_0^{-1}}{\sqrt{r}}\Big(\int_0^t\int |w_y|^2dyd\tau\Big)^\frac12\rightarrow 0, \quad \text{as}~r\rightarrow\infty,
\end{align*}
where $c_0=e^{-\frac{2\sqrt{2}}{\lambda}\sqrt{\mathcal E_0}\|\rho_0\|_1}$.
Therefore, \eqref{2.14} and \eqref{2.15} hold. Consequently, the conclusion follows.
\end{proof}

The following lemma will be the key to show the global in time boundedness of the
entropy.
\begin{lemma}\label{l27}
There exists a positive constant $C$ depending only on $\gamma$, $\mu$, $\lambda$, $K_1$, $\bar{\rho}$, $\mathcal E_0$, $\|P_0\|_2$,
$\|\frac{P_0'}{\sqrt{\rho_0}}\|_2$, $\|\frac{h_0'}{\sqrt{\rho_0}}\|_2$,
 $\|\rho_0^{-\frac{\alpha}{2}}F_0\|_2$,
 $\|\rho_0^{-\frac{\alpha}{2}}G_0\|_2$,
$\|\rho_0^{-\frac{\alpha}{2}}H_0\|_2$, $\|\rho_0^{-\frac{\alpha}{2}}h_0\|_2$,
and $T$ such that
\begin{align}\label{q2.61}
&\sup_{0\le t\le T}\Big\|\Big(\rho_0^{-\frac{\alpha}{2}}F, \rho_0^{-\frac{\alpha}{2}}G, \rho_0^{-\frac{\alpha}{2}}H,
\rho_0^{-\frac{\alpha}{2}}h\Big)\Big\|_2^2
+\int_0^T\Big\|\Big(\rho_0^{-\frac{\alpha+1}{2}}F_y, \rho_0^{-\frac{\alpha+1}{2}}G_y\Big)\Big\|_2^2dt\nonumber\\
&\quad+\int_0^T\left(\Big\|\Big(\rho_0^{-\frac{\alpha}{2}}H^\frac32,
\rho_0^{-\frac{\alpha}{2}}H,
\sqrt{P}\rho_0^{-\frac{\alpha}{2}}H\Big)\Big\|_2^2
+\Big\|\Big(\rho_0^{-\frac{\alpha}{2}}F,
\rho_0^{-\frac{\alpha}{2}}G\Big)\Big\|_\infty^4\right)dt\le C.
\end{align}
%and the constant $C$, viewing as a function of $T$, is continuously increasing with respect
%to $T\in [0, \infty)$.
\end{lemma}
\begin{proof}[Proof]
1. It follows from
$\eqref{1.7}_1$, $\eqref{1.7}_3$, and $\eqref{1.7}_4$ that
\begin{align}\label{2.17}
F_t-\frac{\mu}{J}\left(\frac{F_y}{\rho_0}\right)_y
=-\frac{u_y}{J}F+\frac{w_y}{J}.
\end{align}
For $\delta>0$, set $\rho_\delta=\rho_0+\delta$.
Testing \eqref{2.17} with $\frac{\rho_0}{\rho_\delta^{\alpha+1}}JF\varphi_r^2$, one obtains that
\begin{align}\label{2.58}
\frac12\frac{d}{dt}\int\frac{\rho_0J|F|^2}{\rho_\delta^{\alpha+1}}\varphi_r^2dy+\mu\int\frac{F_y}{\rho_0}
\left(\frac{\rho_0F\varphi_r^2}{\rho_\delta^{\alpha+1}}\right)_ydy=
-\frac12\int u_y\frac{\rho_0F^2\varphi_r^2}{\rho_\delta^{\alpha+1}}dy+\int w_y\frac{\rho_0F\varphi_r^2}{\rho_\delta^{\alpha+1}}dy.
\end{align}
Integration by parts and using Cauchy-Schwarz inequality, we derive from the assumption $|\rho_0'|\le K_1\rho_0^\frac32$ that
\begin{align}\label{2.59}
\int\frac{F_y}{\rho_0}
\left(\frac{\rho_0F\varphi_r^2}{\rho_\delta^{\alpha+1}}\right)_ydy
&=\int\frac{(F_y)^2}{\rho_\delta^{\alpha+1}}\varphi_r^2dy+2\int\frac{F_yF}{\rho_\delta^{\alpha+1}}\varphi_r\varphi_r'dy
+\int F_yF\frac{\rho_0\rho_0'}{\rho_\delta^{\alpha+1}}\left(\frac{1}{\rho_0}-\frac{\alpha+1}{\rho_\delta}\right)\varphi_r^2dy\nonumber\\
&\ge \frac34\int\frac{(F_y)^2}{\rho_\delta^{\alpha+1}}\varphi_r^2dy
-C\int\frac{|F|^2}{\rho_\delta^{\alpha+1}}|\varphi_r'|^2dy
-C\int\frac{\rho_0}{\rho_\delta^{\alpha+1}}|F|^2\varphi_r^2dy.
\end{align}
Substituting \eqref{2.59} into \eqref{2.58} yields that, for $r\ge 1$,
\begin{align}\label{2.60}
&\frac12\frac{d}{dt}\int\frac{\rho_0J|F|^2}{\rho_\delta^{\alpha+1}}\varphi_r^2dy
+\frac{3\mu}{4}\int\frac{|F_y|^2}{\rho_\delta^{\alpha+1}}\varphi_r^2dy\nonumber\\
&\le C\int\left[(|u_y|+|w_y|+1)\frac{\rho_0|F|^2\varphi_r^2}{\rho_0^{\alpha+1}}\right]dy
+C\int\frac{|F|^2}{\rho_\delta^{\alpha+1}}|\varphi_r'|^2dy\nonumber\\
&\le C\int\left[(|u_y|+|w_y|+1)\frac{\rho_0|F|^2\varphi_r^2}{\rho_0^{\alpha+1}}\right]dy
+\frac{C\delta^{-\alpha}}{r^2}\int_{r\le|y|\le 2r}\frac{|F|^2}{\rho_0^{\alpha}}dy.
\end{align}

2. Recalling the definitions of $G$, $H$, and $F$, we obtain that
\begin{align}\label{t2.30}
h_t=\frac{w_y}{J}-h\frac{u_y}{J}
=\frac{1}{\mu}\Big(F-h\Big)-\frac{h}{\lambda}\Big(G+P+\frac{H}{2}\Big),
\end{align}
which yields that
\begin{align*}
H_t=2h\cdot h_t=\frac{2}{\mu}\Big(F-h\Big)\cdot h-\frac{2H}{\lambda}\Big(G+P+\frac{H}{2}\Big).
\end{align*}
Thus, we have
\begin{align}\label{2.30}
H_t+\frac{H^2}{\lambda}+\frac{2H}{\mu}+\frac{2HP}{\lambda}=\frac{2}{\mu}F\cdot h-\frac{2HG}{\lambda}.
\end{align}
Testing \eqref{2.30} with $\frac{\rho_0JH}{\rho_\delta^{\alpha+1}}\varphi_r^2$ and using $\eqref{1.7}_1$, one deduces that
\begin{align*}
&\frac12\frac{d}{dt}\int \rho_0\rho_\delta^{-(\alpha+1)}JH^2\varphi_r^2dy+\int\Big(\frac{1}{\lambda}JH^3
+\frac{2}{\mu}JH^2
+\frac{2}{\lambda}JPH^2\Big)\rho_0\rho_\delta^{-(\alpha+1)}\varphi_r^2dy\nonumber\\
&=\frac12\int\rho_0\rho_\delta^{-(\alpha+1)}u_yH^2\varphi_r^2dy+\int\Big(\frac{2}{\mu}JHF\cdot h-\frac{2}{\lambda}JGH^2\Big)
\rho_0\rho_\delta^{-(\alpha+1)}\varphi_r^2dy\nonumber\\
&=\frac{1}{2\lambda}\int JH^2\Big(G+p+\frac{H}{2}\Big)\rho_0\rho_\delta^{-(\alpha+1)}\varphi_r^2dy
+\int\Big(\frac{2}{\mu}JHF\cdot h-\frac{2}{\lambda}JGH^2\Big)
\rho_0\rho_\delta^{-(\alpha+1)}\varphi_r^2dy,
\end{align*}
which yields that
\begin{align}\label{2.61}
&\frac12\frac{d}{dt}\int JH^2\frac{\rho_0}{\rho_\delta^{\alpha+1}}\varphi_r^2dy+\int\Big(\frac{3}{4\lambda}JH^3
+\frac{2}{\mu}JH^2
+\frac{3}{2\lambda}JPH^2\Big)\frac{\rho_0}{\rho_\delta^{\alpha+1}}\varphi_r^2dy\nonumber\\
&=\frac{2}{\mu}\int JHF\cdot h\frac{\rho_0}{\rho_\delta^{\alpha+1}}\varphi_r^2dy
-\frac{3}{2\lambda}\int JGH^2\frac{\rho_0}{\rho_\delta^{\alpha+1}}\varphi_r^2dy\nonumber\\
&\le \frac{3}{8\lambda}\int JH^3\frac{\rho}{\rho_\delta^{\alpha+1}}\varphi_r^2dy+C
\int JH^2\frac{\rho_0}{\rho_\delta^{\alpha+1}}\varphi_r^2dy+C\int|H|J|F|^2\frac{\rho_0}{\rho_\delta^{\alpha+1}}\varphi_r^2dy.
\end{align}

3. By a direct calculation, we deduce that
\begin{align}\label{2.34}
G_t-\frac{\lambda}{J}\Big(\frac{G_y}{\rho_0}\Big)_y
=-\gamma\frac{u_y}{J}G+\frac{2-\gamma}{2}\frac{u_yH}{J}
-(\gamma-1)\mu\Big|\frac{w_y}{J}\Big|^2-\frac{h\cdot w_y}{J}.
\end{align}
Testing \eqref{2.34} with $\frac{\rho_0}{\rho_\delta^{\alpha+1}}JG\varphi_r^2$, one gets from $|\rho_0'|\le K_1\rho_0^\frac32$ that, for $r\ge 1$,
\begin{align}\label{2.62}
&\frac12\frac{d}{dt}\int \frac{\rho_0}{\rho_\delta^{\alpha+1}}JG^2\varphi_r^2d
+\lambda\int\frac{(G_y)^2\varphi_r^2}{\rho_\delta^{\alpha+1}}dy\nonumber\\
&=\Big(\frac12-\gamma\Big)\int  \frac{\rho_0}{\rho_\delta^{\alpha+1}}u_yG^2\varphi_r^2dy
-2\lambda\int\frac{\rho_0G_y}{\rho_\delta^{\alpha+1}}G\varphi_r\varphi_r'dy
-\lambda\int \frac{GG_y}{\rho_0}\frac{\rho_0\rho_0'}{\rho_\delta^{\alpha+1}}
\left(\frac{1}{\rho_0}-\frac{\alpha+1}{\rho_\delta}\right)\varphi_r^2dy\nonumber\\
&\quad-\frac{2-\gamma}{2\lambda}\int \frac{\rho_0}{\rho_\delta^{\alpha+1}}\Big(JG+JP+\frac{JH}{2}\Big)HG\varphi_r^2dy
-\frac{\gamma-1}{\mu}\int\frac{\rho_0}{\rho_\delta^{\alpha+1}}\Big|\sqrt{J}F-\frac{1}{4\pi}\sqrt{J}h\Big|^2G\varphi_r^2dy\nonumber\\
&\quad-\frac{1}{\mu}\int\frac{\rho_0}{\rho_\delta^{\alpha+1}} h\cdot(JF-Jh)G\varphi_r^2dy\nonumber\\
&\le \frac{\lambda}{4}\int\frac{(G_y)^2\varphi_r^2}{\rho_\delta^{\alpha+1}}dy
+C\int\frac{G^2}{\rho_\delta^{\alpha}}|\varphi_r'|^2dy+C\int(1+|u_y|+|H|+|P|)\frac{\rho_0}{\rho_\delta^{\alpha+1}}G^2\varphi_r^2dy
\nonumber\\
&\quad+\frac{3}{4\lambda}\int JP H^2\frac{\rho_0}{\rho_\delta^{\alpha+1}}\varphi_r^2dy
+C\int(1+|G|)\frac{\rho_0}{\rho_\delta^{\alpha+1}}JH^2\varphi_r^2dy
+C\int\frac{\rho_0}{\rho_\delta^{\alpha+1}}J|F|^2\varphi_r^2dy\nonumber\\
&\le \frac{\lambda}{4}\int\frac{(G_y)^2\varphi_r^2}{\rho_\delta^{\alpha+1}}dy
+\frac{C\delta^{-\alpha}}{r^2}\int_{r\le |y|\le 2r}\frac{G^2}{\rho_0^\alpha}dy+
C\int(1+|u_y|+|H|+|P|)\frac{\rho_0}{\rho_\delta^{\alpha+1}}G^2\varphi_r^2dy
\nonumber\\
&\quad+\frac{3}{4\lambda}\int JP H^2\frac{\rho_0}{\rho_\delta^{\alpha+1}}\varphi_r^2dy
+C\int(1+|G|)\frac{\rho_0}{\rho_\delta^{\alpha+1}}JH^2\varphi_r^2dy
+C\int\frac{\rho_0}{\rho_\delta^{\alpha+1}}J|F|^2\varphi_r^2dy,
\end{align}
which together with \eqref{2.60} and \eqref{2.61} implies that
\begin{align}
&\frac12\frac{d}{dt}\int\frac{\rho_0J}{\rho_\delta^{\alpha+1}}(|F|^2+|G|^2+H^2)\varphi_r^2dy
+\frac{3\mu}{4}\int\frac{(F_y)^2}{\rho_\delta^{\alpha+1}}\varphi_r^2dy
+\frac{3\lambda}{4}\int\frac{(G_y)^2\varphi_r^2}{\rho_\delta^{\alpha+1}}dy\nonumber\\
&\quad+\int\Big(\frac{3}{8\lambda}JH^3+\frac{2}{\mu}JH^2
+\frac{3}{4\lambda}JPH^2\Big)\frac{\rho_0}{\rho_\delta^{\alpha+1}}\varphi_r^2dy\nonumber\\
&\le C\int\Big[(1+H+|u_y|+|w_y|)\frac{\rho_0}{\rho_\delta^{\alpha+1}}J|F|^2\varphi_r^2\Big]dy
+C\int\Big[(1+|G|)\frac{\rho_0}{\rho_\delta^{\alpha+1}}JH^2\varphi_r^2\Big]dy\nonumber\\
&\quad+C\int\Big[(1+|u_y|+|H|+|P|)\frac{\rho_0}{\rho_\delta^{\alpha+1}}JG^2\varphi_r^2\Big]dy
+\frac{C\delta^{-\alpha}}{r^2}\int_{r\le |y|\le 2r}\frac{G^2}{\rho_0^\alpha}dy\nonumber\\
&\quad+\frac{C\delta^{-\alpha}}{r^2}\int_{r\le|y|\le 2r}\frac{|F|^2}{\rho_0^{\alpha}}dy,
\end{align}
from which, it follows from Sobolev's inequality and Lemma \ref{l21} that $$\int_0^T(\|u_y\|_\infty+\|w_y\|_\infty
+\|H\|_\infty+\|G\|_\infty+\|P\|_\infty)dt\le C,$$
for a positive constant $C$ depending only on $\gamma$, $\lambda$, $\mu$, $\bar\rho$, $\frac{J_0'}{\sqrt{\rho_0}}$,
$\left\|\frac{P_0'}{\sqrt{\rho_0}}\right\|_2$, $\left\|\frac{h_0'}{\sqrt{\rho_0}}\right\|_2$,
$\mathcal E_0$, $\|\rho_0\|_1$, $\|G_0\|_2$, $\|F_0\|_2$, $\|H_0\|_2$, $\|P_0\|_2$, and $T$. Thanks to this, we derive that, for $r\ge 1$,
\begin{align}\label{2.64}
&\sup_{0\le t\le T}\left\|\left(\sqrt{\frac{\rho_0J}{\rho_\delta^{\alpha+1}}}F\varphi_r,
\sqrt{\frac{\rho_0J}{\rho_\delta^{\alpha+1}}}G\varphi_r,
\sqrt{\frac{\rho_0J}{\rho_\delta^{\alpha+1}}}H\varphi_r\right)\right\|_2^2
+\int_0^T\|\rho_\delta^{-\frac{\alpha+1}{2}} F_y\varphi_r\|_2^2dt\nonumber\\
&\quad+\int_0^T\|\rho_\delta^{-\frac{\alpha+1}{2}} G_y\varphi_r\|_2^2dt
+\int_0^T\left\|\left(\sqrt{\frac{\rho_0J}{\rho_\delta^{\alpha+1}}}H^\frac32\varphi_r,
\sqrt{\frac{\rho_0J}{\rho_\delta^{\alpha+1}}}H\varphi_r,
\sqrt{\frac{\rho_0JP}{\rho_\delta^{\alpha+1}}}H\varphi_r\right)\right\|_2^2dt\nonumber\\
&\le C\left\|\left(\sqrt{\frac{\rho_0}{\rho_\delta^{\alpha+1}}}F_0\varphi_r,
\sqrt{\frac{\rho_0}{\rho_\delta^{\alpha+1}}}G_0\varphi_r,
\sqrt{\frac{\rho_0}{\rho_\delta^{\alpha+1}}}H_0\varphi_r\right)\right\|_2^2
+\frac{C\delta^{-\alpha}}{r^2}\int_0^T\int_{r\le|y|\le 2r}\frac{|F|^2}{\rho_0^{\alpha}}dydt\nonumber\\
&\quad+\frac{C\delta^{-\alpha}}{r^2}\int_0^T\int_{r\le|y|\le 2r}\frac{G^2}{\rho_0^{\alpha}}dydt.
\end{align}
Note that the assumption $(\rho_0^{-\frac{\alpha}{2}}F, \rho_0^{-\frac{\alpha}{2}}G)\in L^2(0, T; L^2)$
 implies that the last term of
the right hand side of \eqref{2.64} tends to zero as $r\rightarrow \infty$.
Thus, one can take $r\rightarrow \infty$ first and then let $\delta\rightarrow 0$ in \eqref{2.64} to get that
\begin{align}\label{t2.66}
&\sup_{0\le t\le T}\left\|\left(\sqrt{J}\rho_0^{-\frac{\alpha}{2}}F,
\sqrt{J}\rho_0^{-\frac{\alpha}{2}}G,
\sqrt{J}\rho_0^{-\frac{\alpha}{2}}H\right)\right\|_2^2
+\int_0^T\Big\|\rho_0^{-\frac{\alpha+1}{2}} F_y\Big\|_2^2dt\nonumber\\
&\quad+\int_0^T\left\|\left(\rho_0^{-\frac{\alpha+1}{2}} G_y,
\sqrt{J}\rho_0^{-\frac{\alpha}{2}}H^\frac32,
\sqrt{J}\rho_0^{-\frac{\alpha}{2}}H,
\sqrt{JP}\rho_0^{-\frac{\alpha}{2}}H\right)\right\|_2^2dt\nonumber\\
&\le C\left\|\left(\rho_0^{-\frac{\alpha}{2}}F_0, \rho_0^{-\frac{\alpha}{2}}G_0,
\rho_0^{-\frac{\alpha}{2}}H_0\right)\right\|_2^2.
\end{align}

4. Based on the above estimates, we deduce from Gagliardo-Nirenberg inequality
\begin{align*}
\|v\|_\infty\le C\|v\|_2^\frac12\|v'\|_2^\frac12,\ \text{for}\ v\in H^1(\mathbb{R}),
\end{align*}
and the assumption
$|\rho_0'|\le K_1\rho_0^\frac32$ that
\begin{align}\label{t2.67}
\int_0^T\|\rho_0^{-\frac{\alpha}{2}}F\|_\infty^4dt&\le C\int_0^T\|\rho_0^{-\frac{\alpha}{2}}F\|_2^2
\|(\rho_0^{-\frac{\alpha}{2}}F)_y\|_2^2dt\nonumber\\
&=C\int_0^T\|\rho_0^{-\frac{\alpha}{2}}F\|_2^2\left\|\rho_0^{-\frac{\alpha}{2}}F_y
-\frac{\alpha}{2}\frac{\rho_0'}{\rho_0}\rho_0^{-\frac{\alpha}{2}}F\right\|_2^2dt\nonumber\\
&\le C\int_0^T\|\rho_0^{-\frac{\alpha}{2}}F\|_2^2\left(\|\rho_0^{-\frac{\alpha}{2}}F_y\|_2^2
+\alpha^2K_1^2\|\rho_0^\frac{1-\alpha}{2}F\|_2^2\right)dt\nonumber\\
&\le C\int_0^T\|\rho_0^{-\frac{\alpha}{2}}F\|_2^2\left(\|\rho_0^{-\frac{\alpha+1}{2}}F_y\|_2^2
+\|\rho_0^{-\frac{\alpha}{2}}F\|_2^2\right)dt\nonumber\\
&\le C.
\end{align}
Similarly, we have
\begin{align}\label{t2.68}
\int_0^T\|\rho_0^{-\frac{\alpha}{2}}G\|_\infty^4dt\le C.
\end{align}
Testing \eqref{t2.30} with $\frac{h}{\rho_0^\alpha}$ and integrating the resultant over $\mathbb{R}$, one has
\begin{align*}
\frac{d}{dt}\int\frac{h^2}{\rho_0^\alpha}dy&
=\frac{1}{\mu}\int\Big(F-h\Big)\frac{h}{\rho_0^\alpha}dy
-\frac{1}{\lambda}\int\Big(G+P+\frac{H}{2}\Big)\frac{h^2}{\rho_0^\alpha}dy\nonumber\\
&\le C\|\rho_0^{-\frac{\alpha}{2}}F\|_2\|\rho_0^{-\frac{\alpha}{2}}h\|_2+C(\|G\|_\infty+\|P\|_\infty
+\|H\|_\infty)\|\rho_0^{-\frac{\alpha}{2}}h\|_2^2\nonumber\\
&\le C(1+\|G\|_\infty+\|P\|_\infty
+\|H\|_\infty)\|\rho_0^{-\frac{\alpha}{2}}h\|_2^2+C\|\rho_0^{-\frac{\alpha}{2}}F\|_2^2,
\end{align*}
which along with Gronwall's inequality and Lemma \ref{l21} leads to
\begin{align*}
\sup_{0\le t\le T}\|\rho_0^{-\frac{\alpha}{2}}h\|_2^2\le C,
\end{align*}
for a positive constant $C$ depending only on $\gamma$, $\mu$, $\lambda$, $K_1$, $\bar{\rho}$, $\mathcal E_0$, $\|P_0\|_2$,
$\|\frac{P_0'}{\sqrt{\rho_0}}\|_2$, $\|\frac{h_0'}{\sqrt{\rho_0}}\|_2$,
 $\|\rho_0^{-\frac{\alpha}{2}}F_0\|_2$,
 $\|\rho_0^{-\frac{\alpha}{2}}G_0\|_2$,
$\|\rho_0^{-\frac{\alpha}{2}}H_0\|_2$, $\|\rho_0^{-\frac{\alpha}{2}}h_0\|_2$,
and $T$.
This combined with \eqref{t2.66}, \eqref{t2.67}, and \eqref{t2.68} yields the conclusion.
\end{proof}

Now, we are going to derive the an uniform lower bound and upper bound for entropy $s$ as follows.
\begin{lemma}\label{l28}
There exists a positive constant $C$ depending only on $\gamma$, $\mu$, $\lambda$, $K_1$, $\bar{\rho}$, $\mathcal E_0$, $\|P_0\|_2$,
$\|\rho_0^{-\frac{\alpha}{2}}P_0'\|_2$, $\|\rho_0^{-\frac{\alpha}{2}}h_0'\|_2$,
 $\|\rho_0^{-\frac{\alpha}{2}}F_0\|_2$,
 $\|\rho_0^{-\frac{\alpha}{2}}G_0\|_2$,
$\|\rho_0^{-\frac{\alpha}{2}}H_0\|_2$, $\|\rho_0^{-\frac{\alpha}{2}}h_0\|_2$,
and $T$ such that, if $\alpha\ge 1$,
\begin{align*}
\sup_{0\le t\le T}(\|u\|_{H^1}+\|w\|_{H^1}+\|\theta\|_{H^1})\le C.
\end{align*}
Moreover, if $\alpha\ge \gamma$, we have
\begin{align*}
\sup_{0\le t\le T}\|s\|_\infty\le C.
\end{align*}
\end{lemma}
\begin{proof}
1. It follows from \eqref{1.7}$_2$ that
$$\rho_0u_t=\left(\lambda\frac{u_y}{J}-P-\frac{1}{2}|h|^2\right)_y=G_y.$$
Hence, one has $u_t=\frac{G_y}{\rho_0}$, which along with \eqref{t2.66}, \eqref{h1}, and $\alpha\ge 1$ gives that $u_t\in L^2(0, T; L^2)$. Thanks to this and $u_0\in L^2$, one obtains that $u\in L^\infty(0, T; L^2)$. This combined with Lemma \ref{l21} leads to $u\in L^\infty(0, T; H^1)$. Similarly, due to $\frac{F_y}{\rho_0}\in L^2(0, T; L^2)$ and $\rho_0w_t=\left(\mu\frac{w_y}{J}+h\right)_y=F_y$, we derive from \eqref{t2.66}, \eqref{h1}, and $\alpha\ge 1$ that $w_t\in L^2(0, T; L^2)$, which yields that $w\in L^\infty(0, T; L^2)$. This together with Lemma \ref{l21} implies that $w\in L^\infty(0, T; H^1)$.

2. We show that $\theta\in L^\infty(0, T; H^1)$ provided that $\alpha\ge 1$,
$\theta_0\in H^1$, and $\frac{P_0}{\rho_0}\in H^1$. It follows from Lemma \ref{l27} that
\begin{align}
\begin{cases}\label{2.69}
\frac{G}{\sqrt{\rho_0}}\in L^\infty(0, T; L^2)\cap L^4(0, T; L^\infty), \quad \frac{G_y}{\rho_0}\in L^2(0, T; L^2),\\
\frac{F}{\sqrt{\rho_0}}\in L^\infty(0, T; L^2)\cap L^4(0, T; L^\infty), \quad \frac{F_y}{\rho_0}\in L^2(0, T; L^2),\\
\left(\frac{H}{\sqrt{\rho_0}}, \frac{h}{\sqrt{\rho_0}}\right)\in L^\infty(0, T; L^2).
\end{cases}
\end{align}
Due to $|\rho_0'|\le K_1\rho_0^\frac32$, it holds that
\begin{align*}
\left\|\frac{P_0'}{\rho_0}\right\|_2=\left\|\left(\frac{P_0}{\rho_0}\right)'+\frac{P_0\rho_0'}{\rho_0^2}\right\|_2
\le \left\|\left(\frac{P_0}{\rho_0}\right)'\right\|_2+C\left\|\frac{P_0\sqrt{\rho_0}}{\rho_0}\right\|_2\le C\left(\left\|\left(\frac{P_0}{\rho_0}\right)'\right\|_2
+\left\|\frac{P_0}{\rho_0}\right\|_2\right)<\infty.
\end{align*}
In terms of $G$, $F$ and $H$, we rewrite $\eqref{1.7}_5$ as
\begin{align}\label{2.48}
P_t+\frac{1}{\lambda}\Big(P+\frac{2-\gamma}{2}G+\frac{2-\gamma}{4}H\Big)^2
=\frac{\gamma^2}{4\lambda}\Big(G+\frac{H}{2}\Big)^2
+\frac{\gamma-1}{\mu}|F-h|^2.
\end{align}
one gets from \eqref{2.69} and \eqref{t2.2} that
\begin{align*}
\sup_{0\le t\le T}\left\|\frac{P}{\rho_0}\right\|_2&\le \left\|\frac{P_0}{\rho_0}\right\|_2
+C\int_0^T\left\|\frac{G}{\sqrt{\rho_0}}\right\|_\infty\left\|\frac{G}{\sqrt{\rho_0}}\right\|_2dt
+C\int_0^T\left\|\frac{F}{\sqrt{\rho_0}}\right\|_\infty\left\|\frac{F}{\sqrt{\rho_0}}\right\|_2dt\nonumber\\
&\quad+C\int_0^T\left\|\frac{h}{\sqrt{\rho_0}}\right\|_\infty\left\|\frac{h}{\sqrt{\rho_0}}\right\|_2dt
+C\int_0^T\|h\|_\infty\left\|\frac{h}{\sqrt{\rho_0}}\right\|_\infty\left\|\frac{H}{\sqrt{\rho_0}}\right\|_2dt\nonumber\\
&\le C,
\end{align*}
where we have used the following fact
\begin{align*}
\left\|\frac{h}{\sqrt{\rho_0}}\right\|_\infty^2&\le C\left\|\frac{h}{\sqrt{\rho_0}}\right\|_2
\left\|\left(\frac{h}{\sqrt{\rho_0}}\right)_y\right\|_2\le C\left\|\frac{h}{\sqrt{\rho_0}}\right\|_2\left\|\frac{h_y}{\sqrt{\rho_0}}-\frac12\frac{h\rho_0'}{\rho_0^\frac32}\right\|_2\nonumber\\
&\le C\left\|\frac{h}{\sqrt{\rho_0}}\right\|_2\left\|\frac{h_y}{\sqrt{\rho_0}}\right\|_2
+C\|h\|_\infty\left\|\frac{h}{\sqrt{\rho_0}}\right\|_2.
\end{align*}
This combined with $J\in L^\infty(0, T; L^\infty)$ implies that
\begin{align}\label{lz4}
\theta=\frac{P}{R\rho}=\frac{JP}{R\rho_0}\in L^\infty(0, T; L^2).
\end{align}

3. Differentiating \eqref{2.48} with respect to $y$ and multiplying the resultant by $\frac{P_y}{\rho_0^2}$,
one gets from H\"older's and Cauchy's inequalities that
\begin{align}\label{2.71}
\frac{d}{dt}\left\|\frac{P_y}{\rho_0}\right\|_2^2
&\le C\|(P, G, F, H, h)\|_\infty\left\|\left(\frac{P_y}{\rho}, \frac{G_y}{\rho_0},
\frac{F_y}{\rho_0}, \frac{H_y}{\rho_0}, \frac{h_y}{\rho_0}\right)\right\|_2
\left\|\frac{P_y}{\rho_0}\right\|_2\nonumber\\
&\le C\left\|\left(\frac{G_y}{\rho_0},
\frac{F_y}{\rho_0}\right)\right\|_2^2
+(1+\|(P, F, G, H)\|_\infty^2)\left\|\left(\frac{P_y}{\rho_0}, \frac{h_y}{\rho_0}\right)\right\|_2^2.
\end{align}
Differentiating $\eqref{1.7}_4$ with respect to $y$ yields that
\begin{align}\label{t2.53}
h_{yt}+\frac{h_y}{\mu}+\frac{Ph_y}{\lambda}+\frac{hH_y}{2\lambda}
+\frac{Hh_y}{2\lambda}
=\frac{F_y}{\mu}-\frac{1}{\lambda}(h_yG+hG_y+hP_y).
\end{align}
Multiplying \eqref{t2.53} by $\frac{h_y}{\rho_0^2}$, it follows from Cauchy's inequality that
\begin{align}\label{2.72}
\frac{d}{dt}\left\|\frac{h_y}{\rho_0}\right\|_2^2
&\le C(1+\|G\|_\infty+\|P\|_\infty+\|H\|_\infty)\left(\left\|\frac{h_y}{\rho_0}\right\|_2^2
+\left\|\frac{P_y}{\rho_0}\right\|_2^2\right)
+C\left\|\frac{F_y}{\rho_0}\right\|_2^2+C\left\|\frac{G_y}{\rho_0}\right\|_2^2.
\end{align}
Combining \eqref{2.71} and \eqref{2.72} altogether and applying Gronwall's inequality, one has that
\begin{align}\label{lz5}
\sup_{0\le t\le T}\left\|\left(\frac{P_y}{\rho_0}, \frac{h_y}{\rho_0}\right)\right\|_2
\le C.
\end{align}
Since
\begin{align*}
J_y=\Big(\frac{1}{\lambda}\int_0^t\Big(G_y+P_y+\frac{H_y}{2}\Big)d\tau J_0+J_0'\Big)
\exp\Big\{\frac{1}{\lambda}\int_0^t(G+P+\frac{H}{2})d\tau\Big\},
\end{align*}
we obtain that
\begin{align}\label{lz6}
\sup_{0\le t\le T}\left\|\frac{J_y}{\rho_0}\right\|_2
&\le \left(\frac{1}{\lambda}\int_0^T\left\|\left(\frac{G_y}{\rho_0}, \frac{P_y}{\rho_0}\right)\right\|_2dt
+\frac{2}{\lambda}\int_0^T\|h\|_\infty\left\|\frac{h_y}{\rho_0}\right\|_2dt
+\left\|\frac{J_0'}{\rho_0}\right\|_2\right)\nonumber\\
&\quad\times\exp\left\{\frac{C}{\lambda}\int_0^T\|(G, P, H)\|_\infty dt\right\}\nonumber\\
&\le C.
\end{align}
Noting that
\begin{align*}
\theta_y=\frac{1}{R}\left(\frac{JP_y}{\rho_0}+\frac{J_yP}{\rho_0}-
\frac{\rho_0'JP}{\rho_0^2}\right),
\end{align*}
then we infer from \eqref{lz5}, \eqref{lz6},
$(J, P)\in L^\infty(0, T; L^\infty)$, and the assumption $|\rho_0'|\le K_1\rho_0^\frac32$ that
\begin{align*}
\sup_{0\le t\le T}\|\theta_y\|_2\le C\sup_{0\le t\le T}\left(\|J\|_\infty\left\|\frac{P_y}{\rho_0}\right\|_2
+\|P\|_\infty\left\|\frac{J_y}{\rho_0}\right\|_2
+\|J\|_\infty\left\|\frac{P\rho_0^\frac12}{\rho_0}\right\|\right)\le C.
\end{align*}
This combined with \eqref{lz4} implies that $\theta \in L^\infty(0, T; H^1)$.

4. It remains to prove that $s\in L^\infty(0, T; L^\infty)$ under the assumption that $s_0\in L^\infty$ and $\alpha\ge \gamma$. By the definition of $s$, it suffices to show that $\frac{P}{\rho^\gamma}=\frac{J^\gamma P}{\rho_0^\gamma}$ has uniform positive lower and upper bounds in $\mathbb{R}\times (0, T)$.

Differentiating \eqref{2.48} with respect to $y$ yields that
\begin{align}\label{t2.55}
&P_{yt}+\frac{2}{\lambda}\Big(P+\frac{2-\gamma}{2}G+\frac{2-\gamma}{4}H\Big)\Big(P_y+\frac{2-\gamma}{2}G_y
+\frac{2-\gamma}{4}H_y\Big)\nonumber\\
&=\frac{\gamma^2}{2\lambda}\Big(G+\frac{H}{2}\Big)\Big(G_y+\frac{H_y}{2}\Big)
+\frac{2(\gamma-1)}{\mu}(F-h)\cdot(F_y-h_y).
\end{align}
Multiplying \eqref{t2.55} and \eqref{t2.53} by $\frac{P_y}{\rho_0^\alpha}$ and $\frac{h_y}{\rho_0^\alpha}$, respectively,
adding the resultant together, one obtains from H\"older's inequality and Cauchy-Schwarz inequality that
\begin{align*}
\frac{d}{dt}\left\|\left(\frac{P_y}{\rho_0^\frac{\alpha}{2}},
\frac{h_y}{\rho_0^\frac{\alpha}{2}}\right)\right\|_2^2&\le
C\left\|\left(\frac{G_y}{\rho_0^\frac{\alpha}{2}},
\frac{F_y}{\rho_0^\frac{\alpha}{2}}\right)\right\|_2^2
+(1+\|(P, F, G, H)\|_\infty^2)\left\|\left(\frac{P_y}{\rho_0^\frac{\alpha}{2}},
 \frac{h_y}{\rho_0^\frac{\alpha}{2}}\right)\right\|_2^2\nonumber\\
 &\le C\left\|\left(\frac{G_y}{\rho_0^\frac{\alpha+1}{2}},
\frac{F_y}{\rho_0^\frac{\alpha+1}{2}}\right)\right\|_2^2
+(1+\|(P, F, G, H)\|_\infty^2)\left\|\left(\frac{P_y}{\rho_0^\frac{\alpha}{2}},
 \frac{h_y}{\rho_0^\frac{\alpha}{2}}\right)\right\|_2^2,
\end{align*}
from which, by Gronwall's inequality and Lemmas \ref{l21}--\ref{l27}, we have
\begin{align}\label{t2.74}
\sup_{0\le t\le T}\left\|\left(\frac{P_y}{\rho_0^\frac{\alpha}{2}},
\frac{h_y}{\rho_0^\frac{\alpha}{2}}\right)\right\|_2^2\le C.
\end{align}
Furthermore, it follows from Gagliardo-Nirenberg inequality
and the assumption $|\rho_0'|\le K_1\rho_0^\frac32$ that
\begin{align*}
\left\|\frac{h}{\rho_0^\frac{\alpha}{2}}\right\|_\infty^2
\le C\left\|\frac{h}{\rho_0^\frac{\alpha}{2}}\right\|_2
\left\|\left(\frac{h}{\rho_0^\frac{\alpha}{2}}\right)_y\right\|_2
=C\left\|\frac{h}{\rho_0^\frac{\alpha}{2}}\right\|_2
\left\|\frac{h_y}{\rho_0^\frac{\alpha}{2}}
-\frac{\alpha}{2}\frac{h\rho_0'}{\rho_0^{\frac{\alpha}{2}+1}}\right\|_2
\le C\left\|\frac{h}{\rho_0^\frac{\alpha}{2}}\right\|_2
\left\|\frac{h_y}{\rho_0^\frac{\alpha}{2}}\right\|_2
+C\left\|\frac{h}{\rho_0^\frac{\alpha}{2}}\right\|_2^2,
\end{align*}
which together with \eqref{t2.74} and \eqref{q2.61} yields that
\begin{align}\label{t2.76}
\rho_0^{-\frac{\alpha}{2}}h\in L^\infty(0, T; L^\infty).
\end{align}
%Since $\alpha\ge \gamma$, we deduce from \eqref{q2.61} and \eqref{t2.76}
%that
%\begin{align}\label{q2.77}
% (\rho_0^{-\frac{\gamma}{2}}F, \rho_0^{-\frac{\gamma}{2}}G)\in L^4(0, T; L^\infty),\quad \text{and}
%\quad \rho_0^{-\frac{\gamma}{2}}h\in L^\infty(0, T; L^\infty).
%\end{align}
To show the the boundedness from above of $\frac{J^\gamma P}{\rho_0^\gamma}$, due to the boundness of $J$,
one only needs to verify that $\frac{P}{\rho_0^\gamma}$. Indeed, we derive from $\alpha\ge \gamma$, \eqref{t2.2}, \eqref{q2.61}, and \eqref{t2.76} that
\begin{align*}
\sup_{0\le t\le T}\left\|\frac{P}{\rho_0^\gamma}\right\|_\infty
&\le \left\|\frac{P_0}{\rho_0^\gamma}\right\|_\infty+C\int_0^T\left(\left\|\frac{F}{\rho_0^\frac{\gamma}{2}}\right\|_\infty^2
+\left\|\frac{G}{\rho_0^\frac{\gamma}{2}}\right\|_\infty^2\right)dt
+C\int_0^T(1+\|H\|_\infty)\left\|\frac{h}{\rho_0^\frac{\gamma}{2}}\right\|_\infty^2dt\le C.
\end{align*}
Thus, $\frac{P}{\rho_0^\gamma}$ has a uniform upper bound on $\mathbb{R}\times(0, T)$. Moreover, it follows from $\eqref{1.7}_1$ and $\eqref{1.7}_5$ that
\begin{align*}
(J^\gamma P)_t=(\gamma-1)J^{\gamma-2}\big(\lambda(u_y)^2+\mu|w_y|^2\big)\ge 0.
\end{align*}
Thanks to this, it holds that $J^\gamma(y, t)P(y, t)\ge J_0^\gamma(y)P_0(y)=P_0(y)$, which leads to
\begin{align*}
\inf_{y\in\mathbb{R}, t\in [0, T]}\frac{J^\gamma(y, t)P(y, t)}{\rho_0^\gamma(y)}\ge \inf_{y\in\mathbb{R}}\frac{P_0(y)}{\rho_0^\gamma(y)}>0.
\end{align*}
Hence, $\frac{P}{\rho_0^\gamma}$ has a uniform lower bound in $\mathbb{R}\times(0, T)$. Thus, the desired conclusion follows.
\end{proof}

\section{Proof of Theorem \ref{thm2}}\label{sec3}

With all the {\it a priori} estimates in Section \ref{sec2} at hand, we are now ready to prove Theorem \ref{thm2}.

\begin{proof}[Proof of Theorem \ref{thm2}]

We only prove that \eqref{1.9} holds for any finite $T\in (0, \infty)$,
while the validity of \eqref{1.11} follows from \eqref{1.9} and Lemma \ref{l28}.
By ${\rm (ii)}$ of Theorem \ref{thm1}, there exists a positive time $T$ such that \eqref{1.9} holds.
Denote by $T_\ell$ the maximal time such that \eqref{1.9} holds for any $T\in (0, T_\ell)$.
We claim that $T_\ell=\infty$ and thus the conclusion holds. Assume, by contradiction, that $T_\ell<\infty$.
By Lemma \ref{l27}, the following estimate holds
\begin{align*}
&\sup_{0\le t\le T}\|(\rho_0^{-\frac{\alpha}{2}}F, \rho_0^{-\frac{\alpha}{2}}G, \rho_0^{-\frac{\alpha}{2}}H,
\rho_0^{-\frac{\alpha}{2}}h)\|_2^2
+\int_0^T\|(\rho_0^{-\frac{\alpha+1}{2}}F_y, \rho_0^{-\frac{\alpha+1}{2}}G_y)\|_2^2dt\nonumber\\
&\quad+\int_0^T\|(\rho_0^{-\frac{\alpha}{2}}H^\frac32,
\rho_0^{-\frac{\alpha}{2}}H,
\sqrt{P}\rho_0^{-\frac{\alpha}{2}}H)\|_2^2dt+\int_0^T\left\|\left(\rho_0^{-\frac{\alpha}{2}}F,
\rho_0^{-\frac{\alpha}{2}}G\right)\right\|_\infty^4dt\le C,
\end{align*}
for any $T\in (0, T_\ell)$, where $C$ is a positive constant depending only on $\gamma$, $\mu$, $\lambda$, $K_1$, $\bar{\rho}$, $\mathcal E_0$, $\|P_0\|_2$,
$\|\frac{P_0'}{\sqrt{\rho_0}}\|_2$, $\|\frac{h_0'}{\sqrt{\rho_0}}\|_2$,
 $\|\rho_0^{-\frac{\alpha}{2}}F_0\|_2$,
 $\|\rho_0^{-\frac{\alpha}{2}}G_0\|_2$,
$\|\rho_0^{-\frac{\alpha}{2}}H_0\|_2$, $\|\rho_0^{-\frac{\alpha}{2}}h_0\|_2$,
and $T$, and this constant $C$, viewing as a function of $T$,
is continuously
increasing with respect to $T\in [0, \infty)$. Since $T_\ell$ is finite positive number, the
constant $C$ above is actually independent of $T\in (0, T_\ell)$. It follows from this fact that
\begin{align*}
\left(\frac{F(\cdot, T_\ell)}{\rho_0^\frac{\alpha}{2}(\cdot)},
 \frac{G(\cdot, T_\ell)}{\rho_0^\frac{\alpha}{2}(\cdot)},
 \frac{H(\cdot, T_\ell)}{\rho_0^\frac{\alpha}{2}(\cdot)},
 \frac{h(\cdot, T_\ell)}{\rho_0^\frac{\alpha}{2}(\cdot)}\right)\Bigg|_{t=T_\ell}\in L^2.
\end{align*}
Thanks to this, taking $T_\ell$ as the initial time, by Theorem \ref{thm1}, one can see that
\eqref{1.9} holds for some $T_\ell'>T_\ell$, which contradicts to the definition of $T_\ell$.
This contradiction implies that $T_\ell=\infty$, in other words, \eqref{1.9} holds for any finite time
$T\in (0, \infty)$, the conclusion follows.

We now prove the uniqueness. Denote by $(J, u, w, h, P)$ the difference of these two solutions; that is
\begin{align*}
J=J_1-J_2,\quad u=u_1-u_2, \quad w=w_1-w_2, \quad h=h_1-h_2,\quad P=P_1-P_2.
\end{align*}
Then, $(J, u, w, P, h)$ satisfies the followings
\begin{align}
&\partial_tJ=\partial_yu,\label{3.1}\\
&\rho_0\partial_tu-\lambda\partial_y\left(\frac{\partial_yu}{J_1}\right)
+\partial_y\left(P+\frac{\eta}{J_1}J\right)+h_1\cdot\partial_yh_1-h_2\cdot\partial_yh_2=0,\label{3.2}\\
&\rho_0\partial_tw-\mu\partial_y\left(\frac{\partial_yw}{J_1}\right)+\partial_y\left(\frac{\beta}{J_1}J\right)=\partial_yh,\label{3.3}\\
&\partial_th+\varpi \partial_yu+\xi h-\varpi\beta J
-\frac{\partial_yw}{J_1}+\frac{\beta}{J_1}J=0,\label{3.4}\\
&P_t+\gamma\xi P=\chi(\partial_yu-\eta J)+\zeta(\partial_yw-\beta J),\label{3.5}
\end{align}
where
\begin{align*}
&\eta(y, t)=\frac{\partial_yu_2}{J_2}, \quad \beta(y, t)=\frac{\partial_yw_2}{J_2},
\quad \xi(y, t)=\frac{\partial_yu_1}{J_1},\\
&\varpi(y, t)=\frac{h_2}{J_1},\quad \zeta(y, t)=(\gamma-1)\mu\left(\frac{\partial_yw_1}{J_1^2}+\frac{\partial_yw_2}{J_1J_2}\right), \\
&\chi(y, t)=\left[(\gamma-1)\lambda\left(\frac{\partial_yu_1}{J_1^2}+\frac{\partial_yu_2}{J_1J_2}\right)
-\gamma\frac{P_2}{J_1}\right].
\end{align*}
Due to the regularities of $(J_i, u_i, w_i, P_i, h_i)$, $i=1, 2$, it holds that
\begin{align}\label{3.6}
(\eta, \beta, \xi)\in L^2(0, T; L^2),\quad \varpi\in L^\infty(0, T; L^2), \quad
(\eta, \beta, \xi, \varphi, \zeta, \chi)\in L^2(0, T; L^\infty).
\end{align}

Multiplying \eqref{3.1}, \eqref{3.2}, \eqref{3.3}, \eqref{3.4}, and \eqref{3.5}, respectively, by $J\varphi_r^2$, $u\varphi_r^2$, $w\varphi_r^2$,
$h\varphi_r^2$, and $P\varphi_r^2$, summing the resultants up, and integrating
over $\mathbb{R}$, one deduces that
\begin{align*}
&\frac12\frac{d}{dt}\int(J^2+\rho_0u^2+\rho_0w^2+h^2+P^2)\varphi_r^2dy
+\lambda\int\frac{(\partial_yu)^2}{J_1}\varphi_r^2dy
+\mu\int\frac{(\partial_yw)^2}{J_1}\varphi_r^2dy\nonumber\\
&=\int\left[\partial_y uJ+\left(P+\frac{\eta}{J_1}J\right)\partial_yu+\frac{\beta}{J_1}J\partial_yw
+\frac{1}{2}(h_1+h_2)h\partial_yu\right]\varphi_r^2dy\nonumber\\
&\quad+\int\left[-h\partial_yw-\varpi\partial_yuh
-\xi h^2+\varpi\beta Jh+\frac{\partial_yw}{J_1}h
-\frac{\beta}{J_1}Jh-\gamma\xi P^2
+\chi(\partial_yu-\eta J)P
\right]\varphi_r^2dy\nonumber\\
&\quad+\int\zeta(\partial_yw-\beta J)P\varphi_r^2dy
+2\int\left(P+\frac{\eta}{J_1}J-\lambda\frac{\partial_yu}{J_1}+\frac{1}{2}|h_1|^2
-\frac{1}{2}|h_2|^2\right)u\varphi_r\varphi_r'dy\nonumber\\
&\quad-2\int\left(\mu\frac{\partial_yw}{J_1}+h\right)w\varphi_r\varphi_r'dy\nonumber\\
&\le \frac{\lambda}{2}\int\frac{(\partial_yu)^2}{J_1}\varphi_r^2dy
+\frac{\mu}{2}\int\frac{(\partial_yw)^2}{J_1}\varphi_r^2dy
+C\int(J^2+\eta^2J^2+\beta^2J^2+h_1^2h^2)\varphi_r^2dy\nonumber\\
&\quad+C\int(P^2+h^2+h_2^2h^2+\varpi^2h^2+\xi h^2+\xi P^2+\chi^2P^2+\zeta^2P^2)\varphi_r^2dy\nonumber\\
&\quad+C\int(|P|+|\eta||J|+|\partial_y|+|h_1|^2+|h_2|^2)|u||\varphi_r'|dy
+C\int(|\partial_yw|+|h|)|w||\varphi_r'|dy\nonumber\\
&\le  \frac{\lambda}{2}\int\frac{(\partial_yu)^2}{J_1}\varphi_r^2dy
+\frac{\mu}{2}\int\frac{(\partial_yw)^2}{J_1}\varphi_r^2dy+C\int(|\partial_yw|+|h|)|w||\varphi_r'|dy\nonumber\\
&\quad+C(1+\|(\eta, \beta, \xi, \chi, h_1, h_2, \varpi, \zeta)\|_\infty^2)\int(J^2+P^2+h^2)\varphi_r^2dy\nonumber\\
&\quad+C\int(|P|+|\eta||J|+|\partial_yu|+|h_1|^2+|h_2|^2)|u||\varphi_r'|dy,
\end{align*}
which combined with Gronwall's inequality and \eqref{3.6} yields that
\begin{align}\label{3.7}
&\sup_{0\le t\le T}\int(J^2+\rho_0u^2+\rho_0w^2+h^2+P^2)\varphi_r^2dy\nonumber\\
&\le \underbrace{C\int_0^T\int(|\partial_yw|+|h|)|w||\varphi_r'|dydt
+C\int_0^T\int(|P|+|\eta||J|+|\partial_yu|+|h_1|^2+|h_2|^2)|u||\varphi_r'|dydt}_{Q_r}.
\end{align}
It suffices to show that $Q_r$ tends to zero as $r\rightarrow \infty$.

For $t\in (0, T)$, choosing $\xi(t), \eta(t)\in (-1, 1)$ such that
\begin{align*}
u^2(\xi(t), t)\le \frac{2\int_{-1}^1\rho_0u^2dz}{\int_{-1}^1\rho_0dz}
\ \ \ \ \text{and}\ \ \ \ |w|^2(\eta(t), t)\le \frac{2\int_{-1}^1\rho_0|w|^2dz}{\int_{-1}^1\rho_0dz}. 
\end{align*}
Then we get that
\begin{align*}
|u(y, t)|=\Big|u(\xi(t), t)+\int_{\xi(t)}^yu_y(z, t)dz\Big|
\le C\|\sqrt{\rho_0}u\|_2+\|u_y\|_2^\frac12(|y|+1)^\frac12, \quad \forall y\in\mathbb{R},
\end{align*}
which yields that, for any $r\ge 1$,
\begin{align}\label{3.8}
|u(y, t)|\le C\Big(\|\sqrt{\rho_0}u\|_2+\|u_y\|_2^\frac12\sqrt{r}\Big), \quad \forall r\le |y|\le 2r.
\end{align}
Similarly, we have
\begin{align}\label{3.9}
|w(y, t)|\le C\Big(\|\sqrt{\rho_0}w\|_2+\|w_y\|_2^\frac12\sqrt{r}\Big), \quad \forall r\le |y|\le 2r.
\end{align}
It follows from $\eqref{1.7}_1$ that
\begin{align*}
\sup_{0\le t\le T}\|J\|_2\le \int_0^T\|\partial_yu\|_2dt<\infty.
\end{align*}
This together with \eqref{3.5} yields that
\begin{align*}
P=e^{-\gamma\int_0^t\xi d\tau}\int_0^te^{\gamma\int_0^\tau\xi d\tau'}
\big[\chi(\partial_yu-\eta J)+\zeta(\partial_yw-\beta J)\big]d\tau
\end{align*}
from which, by Cauchy-Schwarz inequality and \eqref{3.6} that
\begin{align}\label{3.10}
\sup_{0\le t\le T}\|P\|_2&\le Ce^{\gamma\int_0^T\|\eta\|_\infty dt}
\int_0^T\big[\|\chi\|_\infty(\|\partial_yu\|_2+\|\eta\|_\infty\|J\|_2)+\|\zeta\|_\infty(\|\partial_yw\|_2
+\|\beta\|_\infty\|J\|_2)\big]dt\nonumber\\
&\le C\int_0^T(\|\chi\|_\infty^2+\|\partial_yu\|_2^2+\|\partial_yw\|_2^2+\|\eta\|_\infty^2+\|\beta\|_\infty^2+1)dt<\infty.
\end{align}
Similarly, by virtue of \eqref{3.6} and Cauchy-Schwarz inequality, we have
\begin{align}\label{3.11}
\sup_{0\le t\le T}\|h\|_2&\le Ce^{\int_0^T\|\xi\|_\infty dt}\int_0^T(\|\varpi\|_\infty\|\partial_yu\|_2
+\|\varpi\|_\infty\|\beta\|_\infty\|J\|_2+\|\partial_yw\|_2+\|\beta\|_\infty\|J\|_2)dt\nonumber\\
&\le C\int_0^T(\|\varpi\|_\infty^2+\|\partial_yu\|_2^2+\|\beta\|_\infty^2+\|\partial_yw\|_2^2
+1)dt<\infty.
\end{align}
Thanks to Lemma \ref{l21}, \eqref{3.8}, \eqref{3.9}, \eqref{3.10}, and \eqref{3.11},
 noticing that $|\varphi_r'|\le \frac{C}{r}$,
one gets from H\"older's inequality that
\begin{align*}
Q_r&=C\int_0^T\int_{r\le |y|\le 2r}(|P|+|\eta||J|+|\partial_yu|+|h_1|^2+|h_2|^2)|u||\varphi_r'|dydt\nonumber\\
&\quad+C\int_0^T\int_{r\le |y|\le 2r}(|\partial_yw|+|h|)|w||\varphi_r'|dydt\nonumber\\
&\le \frac{C}{\sqrt{r}}\int_0^T\int_{r\le |y|\le 2r}(|P|+|\eta|+|\partial_yu|+|h_1|^2+|h_2|^2)dydt\nonumber\\
&\quad+\frac{C}{\sqrt{r}}\int_0^T\int_{r\le |y|\le 2r}(|\partial_yw|+|h|)dydt\nonumber\\
&\le \frac{C}{\sqrt{r}}\int_0^T\int_{r\le |y|\le 2r}(|h_1|^2+|h_2|^2)dydt
+C\left(\int_0^T\int_{r\le |y|\le 2r}(|h|^2+|\partial_yw|^2)dydt\right)^\frac12\nonumber\\
&\quad+C\left(\int_0^T\int_{r\le |y|\le 2r}(|P|^2+|\eta|^2+|\partial_yu|^2)dydt\right)^\frac12,
\end{align*}
for any $r\ge 1$, which, together with the fact that $P$, $h$, $h_1$, $h_2$, $\partial_yu$, and $\partial_yw\in L^2(\mathbb{R}\times(0, T))$,
implies
\begin{align*}
Q_r\rightarrow 0, \quad \text{as}~r\rightarrow\infty.
\end{align*}
Thus, we have the conclusion by taking $r\rightarrow \infty$ in \eqref{3.7}. The conclusion follows.
\end{proof}

\end{document}